\documentclass{amsart}

\usepackage{graphicx}

\newtheorem{theorem}{Theorem}
\newcommand{\bt}{\begin{theorem}}
\newcommand{\et}{\end{theorem}}
\newtheorem{lemma}{Lemma}
\newcommand{\bl}{\begin{lemma}}
\newcommand{\el}{\end{lemma}}
\newtheorem{corollary}{Corollary}
\newcommand{\bc}{\begin{corollary}}
\newcommand{\ec}{\end{corollary}}
\newtheorem{conjecture}{Conjecture}
\newcommand{\bconj}{\begin{conjecture}}
\newcommand{\econj}{\end{conjecture}}
\newtheorem{problem}{Problem}
\newcommand{\bprob}{\begin{problem}}
\newcommand{\eprob}{\end{problem}}
\newcommand{\beq}{\begin{equation}}
\newcommand{\eeq}{\end{equation}}
\newcommand{\benum}{\begin{enumerate}}
\newcommand{\eenum}{\end{enumerate}}
\newcommand{\N}{\ensuremath{ \mathbf N }}
\newcommand{\Z}{\ensuremath{\mathbf Z}}

\newcommand{\R}{\ensuremath{\mathbf R}}

\newcommand{\mcb}{\ensuremath{ \mathcal B}}

\newcommand{\mcx}{\ensuremath{ \mathcal X}}

\newcommand{\mbx}{\ensuremath{ \mathbf x}}

\newcommand{\bmat}{\left(\begin{matrix}}
\newcommand{\emat}{\end{matrix}\right)}
\newcommand{\bsmallmat}{\left(\begin{smallmatrix}}
\newcommand{\esmallmat}{\end{smallmatrix}\right)}

\DeclareMathOperator{\support}{\text{support}}

\title{Triangular gaps in the most frequent sizes of $hA$ for $|A|=4$}

\author{Steven Senger}
\address{Department of Mathematics\\Missouri State University\\ Springfield, MO 65809}
\email{stevensenger@missouristate.edu}

\date{\today}

\begin{document}
\maketitle

\begin{abstract}
We explain the triangular gaps observed experimentally in the most popular sizes of the $h$-fold iterated sumset, $hA,$ when $A$ is a randomly chosen four-element subset of the first $q$ natural numbers, for $q$ much larger than $h.$
\end{abstract}

\section{Introduction}
It is well-known that for sufficiently large $q,$ ``most" subsets of $[1..q]$ are $B_h$-sets for parameters $h$ taken to be much smaller than $q.$ Mel Nathanson made this precise in \cite{MelBh}. At the 2025 meeting of the Combinatorial and Additive Number Theory (CANT) conference, he also observed computationally that for a fixed $q,$ the most frequent sizes $hA,$ where $A$ is a four-element subset of $[1..q],$ were separated by consecutive triangular numbers. See his work in \cite{Nathanson2025} and see \cite{OBryant2025} by Kevin O'Bryant for significant insight into this problem and related problems.

The primary goal of this note is to offer a relatively simple explanation for the triangular gap phenomenon, based on some combinatorial estimates. We also present a few related arguments which may be of independent interest. 

\subsection{Notation}

When $u,v\in \Z,$ define the \emph{integer interval} 
\[
[u..v] = \{n \in \Z: u \leq n \leq v \}\text{ and } [v]:=[1..v].
\]

Let $A$ be a finite set of integers and let $hA:=A+A+\dots+A$ with $h$ copies of $A$ to be the $h$-fold sumset of $A$. Notice that if $A \subseteq [a,b]$, then $hA \subseteq [ha,hb]$. Next, use $\binom{[q]}{k}$ to denote the set of all $k$-element subsets of the integer interval $[1..q]$.
Then
\[
\left|\binom{[q]}{k}\right| = \binom{q}{k} = \frac{q^k}{k!} + O\left( q^{k-1}\right).
\]

Let $\mcx_{h,k}$ be the set of all $k$-tuples $\mbx = (x_1,\ldots, x_k)$ 
of nonnegative integers such that $\sum_{i=1}^k x_i = h$. We define the number $M_{h,k}$ to be
\[
M_{h,k}:=|\mcx_{h,k}| = \binom{h+k-1}{k-1} = \frac{h^{k-1}}{(k-1)!} + O_k\left( h^{k-2}\right).
\]
The \emph{support} of the vector $\mbx = (x_1,\ldots, x_k) \in \mcx_{h,k}$ 
is the set 
\[
\support(\mbx) = \{i \in [1..k]: x_i \geq 1\}. 
\] 

For $A = \{a_1,\ldots, a_k\} \in \binom{[q]}{k}$ and $n \in \Z$,   the 
\emph{representation function} $r_{A,h}(n)$ counts the number 
of representations of $n$ as a sum of $h$ elements of $A$. 
Equivalently, writing $\vec{A} = (a_1,\ldots, a_k) \in \N_0^k$, we have 
\[
r_{A,h}(n) = \left|\left\{ \mbx \in \mcx_{h,k}: \mbx\cdot \vec{A} = n \right\}\right|.
\]

The set $A \in \binom{[q]}{k}$ is a $B_h$-set if $r_{A,h}(n) = 0$ or 1 for all integers $n$.
Let $\mcb_{h,k}(q)$ be the set of all $k$-element $B_h$-sets in the interval $[1..q]$.
For all positive integers $k$ we have  
\[
\binom{[q]}{k} = \mcb_{1,k}(q) \supseteq \cdots \supseteq \mcb_{h,k}(q) \supseteq \mcb_{h+1,k}(q) \supseteq \cdots.
\]

One central object of our study will be the set of $B_h$-sets that are not $B_{h+1}$-sets, which we denote by $\mcb^*_{h,k}(q).$ That is,
\[\mcb^*_{h,k}(q) = \mcb_{h,k}(q) \setminus \mcb_{h+1,k}(q).\]

Given sets $A,B\subseteq \N,$ if $a,b,c,d\in A$ are (not necessarily distinct) elements chosen so that $\{a,b\}\neq\{c,d\},$ and $a+b=c+d,$ then we notice that the sumset $A+A$ will be strictly smaller than $M_{2,|A|}.$ We call this a {\it collision} in $A+A$. We similarly define collisions in $hA$ to be cases where some element of $hA$ can be written as two sums of $h$ elements of $A$ that are not merely permutations of one another. For example, $a+b+c = 3d$ is a collision, but $a+2b = b+a+b$ is not. 

The ``triangular gaps" appear in the frequency of the sizes of $hA$ because this phenomenon happens at different values of $h$ for different sets $A.$ To describe this, we introduce some terminology. If we fix $h\geq 2,$ we will call a potential $h$-fold iterated sumset size {\it $h$-frequent} if it occurs $\Omega(h^{-5}q^3)$ times, or {\it $h$-rare} if it occurs $O(h^{13}q^2)$ times.

\subsection{Main results}
Essentially, the argument is that for appropriate choices of $q$ and $h,$ we can show that the maximum possible sumset size $M_{h,4}$ occurs most often, for $\Theta(q^4)$ choices of $A\in\binom{[q]}{4},$ and that the smaller iterated sumset sizes occur frequently ($\Theta(q^3)$ times) when they are a triangular gap away from the previous frequent size, and rarely ($O(q^2)$ times) otherwise. We do this by giving upper and lower bounds on the sizes of $B_{\ell,4}^*(q)$ for relevant choices of $\ell,$ as well as giving upper and lower bounds on the iterated sumset sizes for subsets of $B_{\ell,4}^*(q)$ that are frequent and rare. Specifically, we show that for most $A\in B_{\ell,4}^*(q),$ we will have $|(h+\ell)A|=M_{h,4}-\binom{\ell+2}{3},$ which gives rise to the triangular gaps, and we will show that the exceptional choices are comparatively uncommon.

The first result is a technical lemma that allows us to estimate the number of $B_h$-sets that fail to be $B_{h+1}$-sets, as well as the number of such sets that have more than one collision. Specifically, it tells us that for sets in $\mathcal B_{h,4}^*(q),$ the iterated sumset size $M_{h,4}-1$ is $h$-frequent, while smaller sizes are $h$-rare. Most of the proof relies on highly geometric arguments. These are postponed until Section \ref{lemmaProofSec}.

\begin{lemma}\label{4max}
Given $q$ sufficiently large with respect to $h\geq 2,$ we have
\[(i)~~|\mcb^*_{h,4}(q)|= O\left(h^7q^3\right),\]
\[(ii)~~|\mcb^*_{h,4}(q)|= \Omega\left(h^{-5}q^3\right),\text{ and }\]
\[(iii)~~\left|\{A\in\mcb^*_{h,4}(q): |(h+1)A|\leq M_{h+1,4}-2 \}\right| =O\left(h^{13}q^2\right).\]
\end{lemma}
Specifically, we will use this lemma to prove the following theorem. The first statement in it is a fairly direct corollary, and will be crucial to explaining the triangular gaps. The second statement is already known, as can be seen in far greater generality in \cite{MelUbiq} by Nathanson, but we include it as it is a direct consequence of the rest of the argument.
\begin{theorem}\label{4frequent}
Suppose $q$ is sufficiently large with respect to $h\geq 2.$ For $A\in\binom{[q]}{4}$ iterated sumset size $M_{h,4}$ occurs $\Theta(q^4)$ times and the size $M_{h,4}-1$ occurs $\Omega(h^{-5}q^3)$ and $O(h^{13}q^3)$ times. Moreover, the proportion of $B_h$-sets in $\binom{[q]}{4}$ is increasing in $q.$
\end{theorem}

To see how these results relate to the observed triangular gaps between popular iterated sumset sizes, we prove the following. It shows that if a $B_h$-set fails to be a $B_{h+1}$-set, then we have a straightforward upper bound for the sizes of subsequent iterated sumsets.

\begin{lemma}\label{gapLemma}
For $A\in \mathcal B_{h,4}^*(q),$ we have $|(h+\ell)A|\leq M_{h+\ell,4}-{\ell+2 \choose 3}.$
\end{lemma}

We combine this lemma with some of the arguments in the proof of Theorem \ref{4frequent} to obtain the following quantitative characterization of the triangular gap phenomenon.

\begin{theorem}\label{1and3}
Suppose $q$ is sufficiently large with respect to $h \geq 2,$ and $\ell < h.$ The numbers of elements of $\binom{[q]}{4}$ with sizes of the form $M_{h,4}-{\ell+2 \choose 3},$ are $\Omega(h^{-5}q^3)$, and the numbers of elements with sizes between those are $O(h^{13}q^2).$
\end{theorem}

This gives us that for appropriate choices of $h,\ell$ and $q,$ the iterated sumset sizes
\[M_{h,4}, M_{h,4}-1, M_{h,4}-4, M_{h,4}-10, M_{h,4}-20, \dots, M_{h,4}-{\ell+2 \choose 3}\]
must all be $h$-frequent while the sizes between them must be $h$-rare. As this sequence is just $M_{h,4}$ with successive tetrahedral numbers subtracted, the adjacent terms will have differences equal to the gaps between adjacent tetrahedral numbers, which are consecutive triangular numbers. This shows that the gaps between the largest $h$ sumset sizes that are $h$-frequent must be triangular.

The basic idea will be to prove the main results in Section \ref{theoremProofSec}. This section is largely additive combinatorics, and will assume Lemma \ref{4max} $(i)$ and $(iii)$. In Section \ref{lemmaProofSec}, we give a fairly geometric proof of Lemma \ref{4max} $(i)$ and $(iii)$, and related results. We conclude with a discussion on generalizations to larger set sizes in Section \ref{futureSec}.

\subsection{Acknowledgments}
The author would like to thank Mel Nathanson for simplifying and generalizing numerous parts of this note, particularly for Lemma \ref{repNo}, Kevin O'Bryant for pointing out a crucial error in an earlier draft, and both of them along with Noah Kravitz for their patience and helpful comments, which have greatly improved the quality of this note. He would also like to thank the Vietnam Institute for Advanced Study in Mathematics (VIASM) for the hospitality and for the excellent working conditions.

\section{The additive arguments}\label{theoremProofSec}
In this section, we assume Lemma \ref{4max} $(i)$ and $(iii),$ postponing their proofs until the next section. Here, we prove the main additive combinatorial results. We begin by counting necessary collisions in higher iterated sumsets for any set in $\mathcal B_{h,4}^*(q).$ To illustrate this, we give an explicit example of a collision. We will often denote these by $f(a,b,c,d)$ and $g(a,b,c,d),$ where $f$ and $g$ are distinct linear functions whose coefficients correspond to the entries of sum vectors from $\mathcal X_{h+1,4}.$

With $h=2,$ and $A:=\{1,2,8,10\},$ we could check directly that $A$ is a $B_2$-set, but $1+1+10 = 2+2+8,$ so $A$ is not a $B_3$-set. We can examine this collision by writing $f(a,b,c,d) = 2a+d,$ and $g(a,b,c,d) = 2b+c.$ Here, $f$ corresponds to $(2,0,0,1)\in\mathcal X_{3,4},$ and $g$ corresponds to $(0,2,1,0)\in\mathcal X_{3,4}.$

Separately, notice that if we consider a set of four elements that form an arithmetic progression, then its iterated sumsets of will exhibit maximally many collisions. With these notions in tow, we proceed with the additive combinatorial arguments.

\subsection{Proof of Lemma \ref{gapLemma}}
\begin{proof}
Suppose $A\in \mathcal B_{h,4}^*(q)$. Let the four elements of $A$ be $a<b<c<d.$ Then $hA$ has maximal size, namely, $|hA|=M_{h,4}.$ Since $A\notin \mathcal B_{h+1,4}(q),$ we have that
\[|(h+1)A|\leq M_{h+1,4}-1.\]
This means that by ignoring repetitions due to mere permutations of terms, there is at least one pair of sums of $h+1$ elements from $A$ that evaluate to the same total. Let the functions $f(a,b,c,d)$ and $g(a,b,c,d)$ represent these sums, so $f(a,b,c,d)=g(a,b,c,d),$ but $f$ and $g$ correspond to distinct sum vectors in $\mathcal X_{h+1,4}.$

Now, when we consider the iterated sumset $(h+2)A,$ the size can be at most $M_{h+2,4}-4,$ because even if the other $(h+2)$-fold sums are as distinct as possible (having minimal collisions otherwise), we must have $f+a=g+a, f+b=g+b, f+c=g+c,$ and $f+d=g+d.$ Further, when we consider the iterated sumset $(h+3)A,$ the size can be at most $M_{h+3,4}-10,$ because even if the other $(h+3)$-fold sums are as distinct as possible, we must have
\[f+(a+a)=g+(a+a), f+(a+b)=g+(a+b), \dots, f+(d+d) = g+(d+d),\]
where there are 10 choices for pairs of elements added to both $f$ and $g.$ In general, when we consider $(h+\ell)A,$ it can have size at most $M_{h+\ell,4}-{\ell+2 \choose 3},$ as we will have exactly ${\ell+2 \choose 3}$ choices for $(\ell-1)$-tuples that yield the same $(h+\ell)$-fold sum when added to $f$ as when added to $g.$ 
\end{proof}
\subsection{Proof of Theorem \ref{4frequent}:}
We will get a handle on $|\mcb_{h+1,4}|$ by noticing that it is just the elements of ${[q] \choose 4}$ that are not in any of the $\mcb_{i,4}^*(q),$ for $i\leq h.$ This gives
\[\mcb_{h+1,4}(q)={[q] \choose 4}\setminus\bigcup_{i=1}^h \mcb_{i,4}^*(q).\]
By definition, the $\mcb_{i,4}^*(q)$ are disjoint, so we can apply Lemma \ref{4max} $(i)$ repeatedly for $i\leq h$ to bound the number of sets in $\mathcal B_{i,4}^*(q).$ So for each choice of $i,$ we remove at most $O\left(i^7q^3\right)$ sets from the total of ${q \choose 4},$ leaving
\begin{equation}\label{BhCount}
|\mcb_{h+1,4}(q)|={q \choose 4} - \sum_{i=1}^{h} O\left(i^7q^3\right) = {q \choose 4}-O\left(h^8q^3\right) =\Theta(q^4)
\end{equation}
sets left over as $B_{h+1}$ sets. So the most frequent size of $(h+1)A$ is $M_{h+1,4}.$

Appealing to Lemma \ref{4max} $(ii)$, we see that there are $\Omega\left(h^{-5}q^3\right)$ sets in $\mcb^*_{h,4}(q)$, each of which have size $\leq M_{h+1,4}-1.$ By Lemma \ref{4max} $(iii)$, we see that there are fewer than $\Theta(h^{13}q^2)$ sets with size strictly smaller than $M_{h+1,4}-1,$ meaning that there are at least $\Omega(h^{-5}q^3)$ sets $A$ with size $M_{h+1,4}-1,$ completing the proof of the first statement.

To prove the second statement, notice that evaluating the count $|\mcb_{h,4}(q)|$ in \eqref{BhCount} for increasing values of $q$ shows that for a fixed $h,$ the proportion of $\binom{[q]}{k}$ comprised by $B_h$-sets is indeed increasing. Specifically, the proportion of $B_h$-sets to the total number of sets in $\binom{[q]}{k}$ is given by
\[\frac{|\mcb_{h,4}(q)|}{\left|\binom{[q]}{k}\right|}=\frac{{q \choose 4}-O\left(h^8q^3\right)}{{q \choose 4}}\geq 1-O\left(h^2q^{-1}\right).\]
For a fixed $h,$ this proportion is clearly increasing as $q$ grows.

\subsection{Proof of Lemma 1 $(ii)$:}
In order to prove Theorem \ref{1and3}, we state a companion result to Lemma \ref{gapLemma} that shows that the upper bounds given there are achieved quite often. To prove this result, we show that for a given $h,$ there are many sets in $\binom{[q]}{4}$ that exhibit the expected behavior. Moreover, this result will imply Lemma \ref{4max} $(ii)$.

\begin{lemma}\label{vecPairLB2}
If $h\geq 2,$ then $|\mathcal B_{h,4}^*(q)|=\Omega\left(h^{-5}q^3\right).$ Moreover, there are $\Omega\left(h^{-5}q^3\right)$ choices of $A\in\mathcal B_{h,4}^*(q)$ giving $|(h+\ell)A| = M_{h+\ell,4}-\binom{\ell+2}{3}$ for all $\ell< h.$
\end{lemma}
\begin{proof}
Here, we select a large family of sets from $\binom{[q]}{4}$ that will live in $\mathcal B_{h,4}^*(q).$ In particular, we will consider sets of the form $\{a,b,c,d\}$ that are $B_h$-sets, but have one collision in the $(h+1)$-fold sumset, namely
\begin{equation}\label{collision}
ha+c=(h+1)b,
\end{equation}
so they are not $B_{h+1}$-sets. Moreover, we will restrict the ranges of $a,b,c,$ and $d$ so that this is the only such equality. After that, we verify that for the subsequent iterated sumsets up to $(h+\ell)A,$ all collisions are consequences of this one.

Specifically, first let $a$ range from $1$ to $q(10h)^{-3}.$ Given a choice of $a,$ let $b$ range from $3ha$ to $q(10h)^{-2}.$ So far, we have $\Omega(q^2h^{-5})$ choices for pairs of $a$ and $b.$ Now, given choices of $a$ and $b,$ we want to satisfy \eqref{collision}, so $c$ is fixed to be $(h+1)b-ha,$ which will be some integer between $hb$ and $q(h+1)(10h)^{-2}.$ Finally, we pick $d$ to be any of the $\frac{q}{100}$ integers between $\frac{99}{100}q$ and $q.$ This gives us $\Omega(h^{-5}q^3)$ possible sets of this form.

Notice that any $\{a,b,c,d\}$ chosen as described above will satisfy $a<b<c<d,$ as well as \eqref{collision}. It is plain to see that
\begin{equation}\label{dTooBig}
(h+1)c<q(h+1)^2(10h)^{-2}<q/5<d,
\end{equation}
so there can be no collisions involving $d.$ Since $a<b<c,$ the only possible equalities arising in a $B_h$-set using those three elements must be of the form $(i+j)b=ia+jc$ for some choices of natural numbers $i$ and $j$ whose sum is $\leq h+1.$ By our choice of $c,$ this gives us
\[(i+j)b=ia+jc=ia+j((h+1)b-ha)=(i-hj)a+(h+1)jb\]
\[\Rightarrow (i-hj)b=(i-hj)a,\]
which implies that $i=h$ and $j=1,$ as $a<b.$

To prove the second statement, we notice that we are already done for $\ell =1,$ and show that this selection of sets $A\in\binom{[q]}{4}$ will satisfy the claimed equality for other $\ell < h.$ To show this, we fix a choice of $A\in\mathcal B_{h,4}^*(q)$ of the type given above. We claim that the only possible collisions in $(h+\ell)A$ must be of the form
\begin{equation}\label{trivialCollision}
ha+c+f(a,b,c,d)=(h+1)b+f(a,b,c,d),
\end{equation}
where $f(a,b,c,d)$ is some linear combination of $a,b,c,$ and $d$ with non-negative coefficients summing to $\ell.$ We call such collisions in $(h+\ell)A$ {\it trivial collisions}. To see this, notice that for any $\ell < h,$ we will get that $(h+\ell)c<d$ by arguing as in \eqref{dTooBig}. So if there were to be a nontrivial collision, it would still need to be among $a,b,$ and $c,$ and not involve $d.$ We now look for any possibly nontrivial collision in $(h+\ell)A$ and show that it must indeed be trivial. Any collision would arise from distinct sum vectors ${\bf x},{\bf y}\in\mathcal X_{h+\ell,4}$ with $x_4=y_4=0$ satisfying ${\bf x}\cdot \vec A = {\bf y}\cdot \vec A,$ which we could also write
\begin{equation}\label{3dot}
x_1a+x_2b+x_3c = y_1a+y_2b+y_3c,
\end{equation}
where $x_1+x_2+x_3 = y_1+y_2+y_3=h+\ell.$
Again, appealing to our choice of $c$, we can rewrite this as
\[x_1a+x_2b+x_3[(h+1)b-ha] = y_1a+y_2b+y_3[(h+1)b-ha]\]
\[(x_1-hx_3)a+(x_2+(h+1)x_3)b = (y_1-hy_3)a+(y_2+(h+1)y_3)b.\]
Now, notice that on each side, the coefficients of $a$ are $<2h.$ Since $2ha< 3ha\leq b,$ we see that in order for this equality to hold, we need the coefficients in $a$ on each side agree. Similarly, the coefficients in $b$ on each side must be the same. So we have
\[x_1-hx_3 = y_1-hy_3 \text{ and } x_2+(h+1)x_3 = y_2+(h+1)y_3.\]
The first equation tells us that
\begin{equation}\label{firstCoord}
x_1-y_1 = h(x_3-y_3),
\end{equation}
while the second implies
\begin{equation}\label{secondCoord}
y_2-x_2 = (h+1)(x_3-y_3),
\end{equation}
If $x_1=y_1,$ then \eqref{firstCoord} would then imply that $x_3=y_3.$ Recalling that $x_4=y_4=0,$ and ${\bf x},{\bf y}\in\mathcal X_{h+\ell,4},$ we would then see that $x_2=y_2,$ and these vectors are not distinct, which is a contradiction. A similar argument yields a contradiction if $x_3=y_3,$ so we proceed assuming $x_1\neq y_1$ and $x_3\neq y_3.$ Without loss of generality, suppose $x_1>y_1.$ So \eqref{firstCoord} and the restriction on the range of possible values for entries of vectors in $\mathcal X_{h+\ell,4}$ (namely $x_j,y_j\in[0..h+\ell]$) tells us that $x_1$ and $y_1$ cannot differ by more than $h,$ and cannot be equal, so $x_1=y_1+h.$ This also implies that $x_3=y_3+1.$ Plugging this into \eqref{secondCoord}, we get
\[y_2-x_2=(h+1)(x_3-y_3)=h+1.\]
This means that $x_2 = y_2-(h+1).$ We next show that this collision must be trivial. To see this, recall that a trivial collision will have the form given in \eqref{trivialCollision}. Now combine the above relationships between the $x_j$ and $y_j$ with \eqref{3dot} to get
\[(y_1+h)a+(y_2-(h+1))b+(y_3+1)c = y_1a+y_2b+y_3c.\]
We manipulate this to get
\[ha+c+(y_1a+y_2b+y_3c)=(h+1)b +(y_1a+y_2b+y_3c).\]
So we see that this collision is of the form given by \eqref{trivialCollision} with $f(a,b,c,d)=(y_1,y_2,y_3,0)=g(a,b,c,d)$, and is therefore a trivial collision.
\end{proof}

\subsection{Proof of Theorem \ref{1and3}:}

So by Theorem \ref{4frequent}, the most likely size of $hA$ is $M_{h,4},$ coming from $B_h$-sets $A.$ Quantitatively, we see that $M_{h,4}$ and $M_{h,4}-1$ are both $h$-frequent sizes. In particular, \eqref{BhCount} guarantees that there are $\Theta(q^4)$ sets with size $M_{h,4},$ and all others occur $O(h^8q^3)$ times altogether.

Notice that any set $A$ with $|hA|$ strictly smaller than $M_{h,4}-1$ will either be in $\mathcal B_{h-1,4}^*(q)$ or not. If it is, then by Lemma \ref{4max} $(iii),$ we know that there are $O(h^{13}q^2)$ different choices with size $M_{h,4}-2$ or $M_{h,4}-3$ (or smaller). Moreover, if $A$ was not in $\mathcal B_{h-1,4}^*(q),$ then $h>2$ and it must have been in $\mathcal B_{h-i,4}^*(q)$ for some natural number $i\in[2,h-1].$ Appealing to Lemma \ref{gapLemma}, we see that $|hA|\leq M_{h,4}-4.$ From this reasoning, we have that $M_{h,4}-2$ and $M_{h,4}-3$ are both $h$-rare sizes. See Figure \ref{fi:histogram} below.

\begin{figure}[h]
\centering
\includegraphics[scale=1.5]{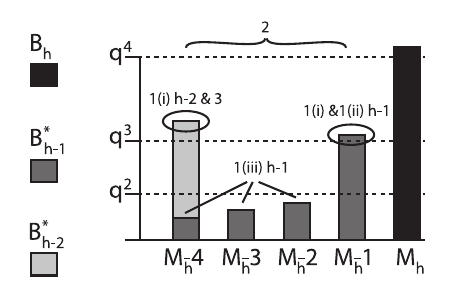}
\caption{Here we use $M_h$ in lieu of $M_{h,4}$ to save space. The numbers indicate which lemmata are used for each estimate for the five largest possible sizes of $hA.$ The size $M_h$ is most frequent, occurring $\Theta(q^4)$ times. The size $M_h-1$ occurs $\Theta(q^3)$ times, by Lemma \ref{4max} $(i)$ and $(ii)$, applied to $h-1.$ We also see how Lemma \ref{4max} gives upper bounds for the $h$-rare sizes, followed by the contribution due to $B_{h-2}^*$ spiking up at $M_h-4,$ as quantified by Lemma \ref{4max} $(i)$ applied to $h-2$ and Lemma \ref{vecPairLB2}. We know the gap must be this wide by Lemma \ref{gapLemma}.}
\label{fi:histogram}
\end{figure}

Now we apply Lemma \ref{vecPairLB2} to see that there are $\Omega(h^{-5}q^3)$ sets in $A\in \mathcal B_{h-2,4}^*(q)$ with size $M_{h,4}-4,$ so we know this size is $h$-frequent. From here we continue by applying Lemma \ref{4max} $(iii)$ again to bound the contribution to the count of sumset sizes from $M_{h,4}-5$ to $M_{h,4}-9$ from elements of $\mathcal B_{h-1,4}^*(q),$ and applying Lemma \ref{vecPairLB2} to see that $M_{h,4}-10$ must be $h$-frequent. We continue this process until $h$ triangular gaps have been guaranteed.


\section{The geometric arguments (proof of Lemma \ref{4max})}\label{lemmaProofSec}

The general strategy will be to consider the space of sets $\binom{[q]}{4} = [1..q]^4$ partitioned into subsets according to the maximal $h$ for which each element is a $B_h$-set. Then, within each of these subsets, we will identify a set of planes corresponding to different sums of elements agreeing. We then show that for sufficiently large $q,$ there are more elements in exactly one of these planes than in many such planes.

We now return to the helpful vector notation given above. Given $A\in\binom{[q]}{4},$ recall that $\vec A$ is a 4-dimensional set vector whose entries are the elements of $A$. Now, given a sum vector ${\bf x} \in \mathcal X_{h,4},$ the dot product $\vec A \cdot {\bf x}$ gives an element of $hA.$ For a given number, $s\in \N,$ the set of sum vectors ${\bf v}\in \R^4$ that have ${\bf x}\cdot {\bf v} = s$ will form a three-dimensional hyperplane in $\R^4,$ which we will call $P_{\bf x}(s),$ or when context is clear, just $P_{\bf x}.$ Moreover, such vectors ${\bf v}$ whose entries are distinct elements of $[1..q]$ correspond to sets in $\binom{[q]}{k}.$ In what follows, we will occasionally consider vectors whose elements are not distinct, leading to some potential inaccuracies. However, these inaccuracies will always be smaller than the main quantities handled.
\subsection{Proof of part $(i)$}
If $A\in \mcb^*_{h,k}(q),$ then for all distinct pairs of sum vectors ${\bf w}, {\bf z}\in\mathcal X_{h,4},$ we will have
\[\vec A \cdot {\bf w} \neq \vec A \cdot {\bf z},\]
but there must exist at least one pair of distinct sum vectors ${\bf x}, {\bf y}\in\mathcal X_{h+1,4}$ so that
\[\vec A \cdot {\bf x} = \vec A \cdot {\bf y} = s\in (h+1)A.\]
Now, the sets of sum vectors ${\bf v}$ that satisfy the equations ${\bf x} \cdot {\bf v}=s$ and ${\bf y} \cdot {\bf v} = s$ determine distinct hyperplanes in $\R^4,$ denoted by $P_{\bf x}(s)$ and $P_{\bf y}(s),$ respectively. So by definition, $\vec A$ must lie on $P_{\bf x}(s)\cap P_{\bf y}(s)$. However, $\vec A$ cannot lie on the intersection of the planes determined by any two sum vectors from $\mathcal X_{h,4}.$
\begin{lemma}\label{ortho}
Given $A\in \mcb^*_{h,4}(q),$ any distinct sum vectors ${\bf x},{\bf y}\in\mathcal X_{h+1,4}$ that have the same dot product with $\vec A$ must have disjoint support (and therefore are also orthogonal). That is, no entry is strictly positive in both ${\bf x}$ and ${\bf y}.$
\end{lemma}
\begin{proof}
To see this, suppose for contradiction that ${\bf x}$ and ${\bf y}$ both have a positive entry in some coordinate. Without loss of generality, suppose that this is the first coordinate, and denote the first entry of $A$ by $a.$ We would then have
\[({\bf x} - (1,0,0,0))\cdot \vec A=({\bf x} \cdot A)-a = ({\bf y} \cdot A)-a = ({\bf y} - (1,0,0,0))\cdot \vec A,\]
but since ${\bf x} - (1,0,0,0)$ and ${\bf y} - (1,0,0,0)$ are both elements of $\mathcal X_{h,4},$ they cannot have the same dot product with $\vec A,$ and we arrive at a contradiction.
\end{proof}

We use this to get the following bound on the number of representations of any sum in $(h+1)A$ when $A\in \mcb^*_{h,k}(q).$
\bl \label{repNo}
If $A \in \mcb^*_{h,k}(q)$, then $r_{A,h+1}(n) \geq 2$ for some $n \in \Z$ 
and 
\[
r_{A,h+1}(n) \leq \left[ \frac{k+1}{2} \right].\]
\el

\begin{proof}
Let $A \in \mcb^*_{h,k}(q)$.  Because $A \notin \mcb_{h+1}(q)$, 
there is an integer $n$ such that $r_{A,h+1}(n) \geq 2$. Let  $n \in (h+1)A$ satisfy $r_{A,h+1}(n) = r\geq 2$.  
If $\mbx_1,\ldots, \mbx_r \in \mcx_{h,k}$ are distinct vectors such that $\mbx_i \cdot \vec{A} =  n$ 
for all $i \in [1,r]$, then by Lemma \ref{ortho}, the supports of the vectors $\mbx_1,\ldots, \mbx_r$ 
are pairwise disjoint.  Moreover, $|\support(\mbx_i)|=1$ for at most one $i \in [1,r],$ because $A$ consists of distinct elements, and so  
\[
2r-1 = 2(r-1)+1 \leq \sum_{i=1}^r |\support(\mbx_i)| \leq k. 
\]
This completes the proof. 
\end{proof} 

Applying Lemma \ref{repNo} for $k=4$ gives the following corollary.
\begin{corollary}\label{3planes}
Given $A\in \mcb^*_{h,4}(q),$ there can be at most two sum vectors from $\mathcal X_{h+1,4}$ that have any fixed dot product with $\vec A.$
\end{corollary}
The next result gives an estimate on how many distinct $(h+1)$-fold sums can occur. We give the set of all possible sums over all relevant sets a name, $\mathcal S_{h+1,4}(q).$ That is
\[\mathcal S_{h+1,4}(q):=\left\{{\bf x} \cdot \vec A :{\bf x}\in\mathcal X_{h+1,4},A\in\binom{[q]}{k}\right\}.\]
\begin{lemma}\label{ddp}
We have $[5h..hn]\subseteq \mathcal S_{h+1,4}(q),$ and moreover,
\[|\mathcal S_{h+1,4}(q)|=\Theta(hq).\]
\end{lemma}
\begin{proof}
For any pair of vectors, where one is a sum vector and the other is a set vector, $({\bf z}, \vec A)\in\mathcal X_{h+1,4}\times \binom{[q]}{k},$ consider the maximum and minimum values of any entry. The entries of the sum vector ${\bf x}$ will all come from $[0..(h+1)],$ and the entries of the set vector $\vec A$ will all come from $[1..q].$ Since both vectors are four-dimensional, the maximum possible dot product of the form ${\bf z} \cdot \vec A$ is no more than $4(h+1)q.$ Next, fix a putative dot product $s\in[5h..hq].$ By the Division Algorithm applied to $s$ and $h,$ there exist $a\in [1..q]$ and $b\in[0..q]$ so that $s=ah+b.$ We will show that each of the $h(q-5)$ distinct values of $s$ considered here could be attained as dot products of the form ${\bf z} \cdot \vec A.$ To see this, we split into cases depending on whether or not $a$ and $b$ are distinct. If $a=b,$ we set ${\bf z} = (h+1,0,0,0),$ and $\vec A=(a,c,d,e),$ for some distinct $c,d,e\in[1..q]\setminus\{a\}.$ If $a\neq b,$ we set ${\bf z} = (h,1,0,0),$ and $\vec A=(a,b,c,d),$ for some distinct $c,d\in[1..q]\setminus\{a,b\}.$ Combining these facts, we see that there are $\Theta(hq)$ total possible distinct dot products of the form ${\bf z} \cdot \vec A.$
\end{proof}

We use this to obtain upper and lower bounds on how many possible sets $A\in\binom{[q]}{k}$ can correspond to vectors in the intersection of two distinct hyperplanes $P_{\bf x}$ and $P_{\bf y}$ coming from distinct sum vectors ${\bf x}, {\bf y}\in\mathcal X_{h+1,4}.$ Given a pair of sum vectors ${\bf x}, {\bf y} \in \mathcal X_{h+1,4}$ with disjoint support, let $T_{x,y}$ denote the number of a sets $A \in\mcb^*_{h,4}(q),$ satisfying
\begin{equation}\label{eqVec}
{\bf x} \cdot \vec A = {\bf y} \cdot \vec A.
\end{equation}

\begin{lemma}\label{vecPairUB}
Given a pair of sum vectors ${\bf x}, {\bf y} \in \mathcal X_{h+1,4}$ with disjoint support, we have that $T_{x,y}=O(hq^3).$
\end{lemma}
\begin{proof}
We will get upper and lower bounds on the number of sets $A=(a,b,c,d)\in \binom{[q]}{4}$ that satisfy \eqref{eqVec}. For a given $s\in S_{h+1,4}(q),$ as above, we define $P_{\bf x}(s)$ as the set of vectors ${\bf v}\in\R^4$ such that ${\bf x} \cdot {\bf v} =s,$ and define $P_{\bf y}(s)$ similarly. As ${\bf x}$ and ${\bf y}$ are linearly independent, we see that $P_{\bf x}(s)\cap P_{\bf y}(s)$ is a plane, and therefore can have no more than $\Theta\left(q^2\right)$ lattice points from $[1..q]^4.$ So there are no more than $\Theta \left(q^2\right)$ choices of $A\in\binom{[q]}{4}$ satisfying \eqref{eqVec} for the dot product $s.$ Lemma \ref{ddp} tells us that there are $\Theta(hq)$ choices for $s\in \mathcal S_{h+1,4}(q).$ Since $\mcb^*_{h,4}(q)\subseteq \binom{[q]}{4},$ this yields the claimed upper bound.
\end{proof}

Now notice that there are $M_{h+1,4}=\Theta\left(h^3\right)$ elements in $\mathcal X_{h+1,4},$ meaning that there are $\Theta\left(h^6\right)$ pairs of distinct sum vectors that can be chosen from $\mathcal X_{h+1,4}.$ Combining this with Lemma \ref{vecPairUB} completes the proof of Lemma \ref{4max} $(i)$.

\subsection{Proof of Lemma \ref{4max} $(ii)$ for small $h$}

While Lemma \ref{vecPairLB2} already implies the statement of Lemma \ref{4max} $(ii),$ we also include a proof of a greater lower bound, but it only holds for small values of $h.$ Notice that Lemma \ref{vecPairLB2} has worse asymptotic dependence on $h,$ but a much wider range of $h.$

\begin{lemma}\label{vecPairLB1}
If $h=2$ or $3,$ there exists a pair of sum vectors ${\bf x}, {\bf y} \in \mathcal X_{h+1,4}$ with disjoint support so that $T_{x,y}=\Omega\left(q^3\right).$
\end{lemma}
\begin{proof}
We are trying to find vectors of ${\bf x},{\bf y}\in \mathcal X_{h+1,4},$ so that there are many choices of sets $A\in\mathcal B_{h,4}^*(q)$ satisfying \eqref{eqVec}. Since any such choice of ${\bf x}$ and ${\bf y}$ will have both vectors coming from $\mathcal X_{h+1,4},$ we know they will be linearly independent. For each $h,$ we call this set of pairs of vectors $\mathcal P_{h,4}.$ Moreover, since they satisfy \eqref{eqVec} for $A\in\mathcal B_h(q),$ we know they have disjoint support by Lemma \ref{ortho}. Since all of the entries are nonnegative, we can phrase this as the vectors being orthogonal. In general we define
\[\mathcal P_{h,k}:=\{\{{\bf x},{\bf y}\}\subseteq \mathcal X_{h,k}^2: {\bf x} \cdot {\bf y} = 0\}.\]
To count the number of pairs in $\mathcal P_{h,4},$ we split into two cases: the case where one vector has support of size one, and the case where both vectors have support of size two.

In the first case, we have 4 choices for ${\bf x},$ a vector with a single entry of $h$. For each of those we now count how many vectors ${\bf y}$ have support outside of the support of ${\bf x}.$ To do this, we need to know how many ways three entries (possibly zero) could sum to $h.$ By stars and bars, we get that there are ${h+2 \choose 2}$ ways for three nonnegative integers to sum to $h.$ However, this is a slight overcount, as every time the stars and bars gives us a single entry of $h$ and two zero entries, we have a pair of single entry vectors, which we are counting twice. So we subtract the ${4 \choose 2} =6$ pairs we are overcounting by to get that the first case has a total of $4{h+2 \choose 2}-6$ pairs of vectors.

In the second case, we need to choose which pair of entries will be supported in which vector. There are ${4 \choose 2} = 6$ ways to choose a pair of entries, but we notice that choosing one pair of entries first implies that we will choose the other pair of entries first in another of these counts, so we divide by 2 to get a total of 3 different ways to partition the four entries into two disjoint pairs of two entries each. For each such partition, we have to have two positive integers that sum to $h,$ which we again calculate to be ${h-1 \choose 1} = h-1,$ by stars and bars, for each pair. So the total count for the second case is $3(h-1)^2.$

Putting these together, we get
\[|\mathcal P_{h,4}| = 4{h+2 \choose 2} -6+ 3(h-1)^2 = 5h^2+1.\]
Moving on, we recall that we can never have ${\bf x} \cdot \vec A = {\bf y}\cdot \vec A,$ for $A\in\mathcal B_{h,4}^*(q),$ when both ${\bf x},{\bf y}\in\mathcal X_{h,4}$ have exactly one nonnegative entry, as the entries in $A$ are distinct, so this would imply $ha=hb$ for some $a\neq b.$ Since there are ${4 \choose 2} = 6$ such pairs accounted for in $\mathcal P_{h,4},$ we are presently more concerned with the size of $\mathcal P_{h,4}',$ which is $\mathcal P_{h,4}$ with the six pairs of single-support vectors removed. So we have
\begin{equation}\label{Pcount}
|\mathcal P_{h,4}'|=|\mathcal P_{h,4}|-6 = 5h^2-5.
\end{equation}
That is, the number of pairs of vectors ${\bf x}$ and ${\bf y}$ in $\mathcal X_{h,4}$ that will satisfy \eqref{eqVec} for some $A\in\mathcal B_{h,4}^*$ is $|\mathcal P_{h,4}'|=5h^2-5.$ So we see that when $h=1,$ there are exactly $P_{1,4}=0$ pairs of vectors satisfying \eqref{eqVec} for any $A\in\binom{[q]}{4}.$ This corresponds to the very uninteresting fact that every set $A\in\binom{[q]}{4}$ being a $B_1$-set.

When $h=2,$ there are $|\mathcal P_{2,4}'|=15$ pairs of vectors $\{{\bf x},{\bf y}\}$ chosen from $\mathcal X_{2,4}$ satisfying \eqref{eqVec} for choices of $A\in\mathcal B_{1,4}^*(q).$ By definition, any set $A$ satisfying \eqref{eqVec} with some appropriate pair ${\bf x}$ and ${\bf y}$ must lie on the set $P_{\bf x}(s)\cap P_{\bf y}(s),$ where $s={\bf x}\cdot \vec A.$ Since all such pairs of ${\bf x}$ and ${\bf y}$ are linearly independent, the intersections $P_{\bf x}\cap P_{\bf y}$ are all planes. Arguing as in the proof of Lemma \ref{vecPairUB}, if $q$ is large enough, then each of these planes has $\Omega(q^2)$ points in it, and there are $\Omega(q)$ choices of $s$ for which this can happen. Putting these together gives us that $\mathcal B_{1,4}^*(q)=\Omega(q^3).$

We follow the same argument for $h=3,$ and get that by definition, for any pair ${\bf x}$ and ${\bf y}$ chosen from $\mathcal P_{3,4}',$ and any pair ${\bf z}$ and ${\bf w}$ chosen from $\mathcal P_{2,4}',$ the plane $P_x\cap P_y$ cannot intersect the plane $P_z\cap P_w$ in more than a line. Since there are only 15 pairs in $\mathcal P_{2,4}',$ and \eqref{Pcount} tells us there are 40 pairs in $\mathcal P_{3,4}',$ so even if some plane determined by a pair of vectors from $\mathcal P_{3,4}',$ could also be determined by a pair of vectors from $\mathcal P_{2,4}',$ we are still guaranteed that there are at least $40-15=25$ new planes determined by pairs of vectors in $\mathcal P_{3,4}',$ each with $\Omega(q^2)$ points, meaning that again, $\mathcal B_{2,4}^*(q)=\Omega(q^3).$

We can run the same argument yet again, but this time, we need to count planes determined by pairs of vectors from $\mathcal P_{4,4}',$ that cannot be determined by pairs of vectors from either $\mathcal P_{2,4}',$ or $\mathcal P_{1,4}'.$ For example, the plane determined by the pair $(2,2,0,0)$ and $(0,0,2,2),$ chosen from $\mathcal P_{4,4}',$ is the same as the plane determined by the pair $(1,1,0,0)$ and $(0,0,1,1)$ chosen from $\mathcal P_{1,4}'.$ However, by again appealing to \eqref{Pcount}, we get that
\[75=|\mathcal P_{4,4}'|> |\mathcal P_{3,4}'|+|\mathcal P_{2,4}'| = 40 + 15 = 55.\]
This tells us that we again have $\mathcal B_{3,4}^*(q)=\Omega(q^3).$ For $h=5,$ this approach will fail, as the union bound overtakes the size of $\mathcal P_{5,4}'.$
\end{proof}

\subsection{Proof of part $(iii)$}
Suppose that a given set vector $\vec A$ lies in the intersection $P_{\bf x}(s) \cap P_{\bf y}(s)$ for some distinct sum vectors ${\bf x}, {\bf y} \in \mathcal X_{h+1,4}$ and $s\in (h+1)A.$ By definition, every pair of sum vectors chosen from $\mathcal X_{h+1,4}$ is linearly independent. So if there is another pair of distinct sum vectors, ${\bf p},{\bf r} \in \mathcal X_{h+1,4}$ and a $t\in (h+1)A,$ so that
\[{\bf p} \cdot \vec A= {\bf r} \cdot \vec A = t,\]
then we can call their respective hyperplanes $P_{\bf p}(t)$ and $P_{\bf r}(t).$ Notice that $t\neq s$ by Corollary \ref{3planes}. Call this set of vectors $V.$ Namely, set
\[V:=\{{\bf p}, {\bf r}, {\bf x}, {\bf y}\}.\]
This gives us the following lemma.

\begin{lemma}\label{4planes}
The intersection of all four hyperplanes
\[P_{\bf x}(s)\cap P_{\bf y}(s)\cap P_{\bf p}(t)\cap P_{\bf r}(t)\]
is at most one line.
\end{lemma}
\begin{proof}
Recall that each pair of sum vectors is linearly independent. We will prove that $V$ has no linearly dependent triples by contradiction. To see this, suppose that we have a linearly dependent triple in $V.$ Since $V$ is comprised of two pairs of sum vectors with disjoint support (by Lemma \ref{ortho}), any triple of vectors chosen from $V$ must have a pair with disjoint support, by the pigeonhole principle. So without loss of generality, suppose ${\bf x}, {\bf y},$ and ${\bf p}$ form a linearly dependent triple. Since ${\bf x}$ and ${\bf y}$ have disjoint support, with at least one of them having at least two nonzero entries, without loss of generality, suppose that ${\bf x}$ has at least two nonzero entries. This means that ${\bf p}$ must have at least three nonzero entries. Recall that ${\bf p}$ and ${\bf r}$ must also have disjoint support, and ${\bf r}$ must have at least one nonzero entry, so ${\bf p}$ can have at most three nonzero entries. Therefore, ${\bf p}$ has exactly three nonzero entries. This means that both ${\bf y}$ and ${\bf r}$ have exactly one nonzero entry, and ${\bf x}$ has exactly two nonzero entries.

Without loss of generality, suppose that for some natural number $j,$ we have
\[{\bf x} = (j,h+1-j,0,0) \text{ and }{\bf y} = (0,0,h+1,0).\]
Recall that ${\bf p}$ is linearly dependent on the pair ${\bf x}$ and ${\bf y}.$ So there must be some rational $\lambda\in(0,1)$ so that 
\[{\bf p} = (\lambda j, \lambda (h+1-j), (1-\lambda)(h+1),0), \text{ and }{\bf r} = (0,0,0,h+1).\]
So for $\vec A=(a,b,c,d)\in P_{\bf x}(s)\cap P_{\bf y}(s) \cap P_{\bf p}(t) \cap P_{\bf r}(t),$ we have
\begin{equation}\label{sEq}
aj+(h+1-j)b =  (h+1)c =s
\end{equation}
and for $t\neq s,$
\[\lambda j a + \lambda (h+1-j) b + (1-\lambda)(h+1)c = (h+1)d=t.\]
But plugging in from \eqref{sEq}, we get
\[\lambda s+ (1-\lambda)s= (h+1)d=t,\]
which contradicts the fact that $s\neq t.$

Therefore we have no linearly dependent triples of vectors in the set $V:=\{{\bf p}, {\bf r}, {\bf x}, {\bf y}\},$ and the intersection of the four relevant hyperplanes cannot be a plane, but could potentially be a line. 
\end{proof}

The two ways that $(h+1)A$ can have size $\leq M_{h+1,4}-2$ are if there is one sum that is achieved at least three different ways or at least two sums that are achieved exactly two different ways. The first situation happens when $\vec A$ is on the intersection of three hyperplanes, $P_{\bf x}(s)\cap P_{\bf y}(s) \cap P_{\bf z}(s),$ which, as we saw above in Corollary \ref{3planes} cannot happen. Notice that on any line, there are $\leq \sqrt{4}q = 2q$ vectors from $\binom{[q]}{4}$, so the second situation happens at most $2q$ times for each of the $M_{h+1,4}^4 =\Theta\left(h^{12}\right)$ quadruples of choices of sum vectors from $\mathcal X_{h+1,4},$ by Lemma \ref{4planes}. By ranging over all possible distinct dot products using Lemma \ref{ddp}, we see that there are $\Theta(hq)$ possible choices for $s.$ Notice that the choices we have made thus far will then fix the dot product $t.$ So in total, there are $\leq 2q\Theta\left(hq\right)\Theta\left(h^{12}\right)$ possibilities for this to happen, completing the proof of Lemma \ref{4max} $(iii)$.

\section{Generalizing to larger $|A|$}\label{futureSec}
We briefly discuss how one could extend the following argument to describe similar phenomena for larger $|A|.$ Suppose $A\in\mcb^*_{h,k}(q)$. Then $|hA|=M_{h,k}.$ Moreover, by reasoning as in the proof of Lemma \ref{gapLemma}, we would have the following estimates for the sizes for successive iterated sumsets of $A:$
\[|(h+1)A| \leq M_{h+1,k}-1,\]
\[|(h+2)A| \leq M_{h+2,k}-k,\]
\[|(h+3)A| \leq M_{h+3,k}-M_{2,k},\]
\[\vdots\]
\[|(h+i)A| \leq M_{h+i,k}-M_{i-1,k}.\]
If a generalized version of Theorem \ref{1and3} with $k > 4$ were to hold, then the successive gaps between the most frequent iterated sumset sizes should have the form:
\[[M_{h+i,k}-M_{i-1,k}]-[M_{h+i-1,k}-M_{i-2,k}]=\]
\[\left[{(h+i)+k-1 \choose k-1}-{(i-1)+k-1 \choose k-1}\right]\]
\[-\left[{(h+i-1)+k-1 \choose k-1}-{(i-2)+k-1 \choose k-1}\right].\]
In particular, if $k=4,$ the differences are tetrahedral numbers, whose differences give us the triangular gaps observed above. If $k=5,$ these gaps take the form of pentatope numbers (figurative numbers based on the four-dimensional simplex). In general, we seem to get the $(k-1)$-dimensional ``champagnerpyramide" numbers (see \cite{Baumann}).

\end{document}